\documentclass{article}
\usepackage{graphicx} 
\usepackage{enumitem}
\usepackage{tikz}
\usepackage{url}

\usepackage[right=4cm]{geometry}

\usepackage{amsmath,amssymb,latexsym, amsthm}
\usepackage{amsfonts}
\usepackage{fullpage}

\usepackage[colorlinks=true,citecolor=blue]{hyperref}

\newcommand{\RR}{\mathbb{R}}

\newcommand{\mr}{\mathrm}

\DeclareMathOperator{\conv}{conv}

\DeclareMathOperator{\HJ}{HJ}
\DeclareMathOperator{\dhj}{dhj}

\newtheorem{thm}{Theorem}[section]
\newtheorem{Q}[thm]{Question}

\newtheorem{con}[thm]{Conjecture}
\newtheorem{lemma}[thm]{Lemma}

\newtheorem{cor}[thm]{Corollary}

\theoremstyle{definition}

\newtheorem{remark}[thm]{Remark}

\usepackage{pgfplots}
\pgfplotsset{compat=1.15}
\usepackage{mathrsfs}
\usetikzlibrary{arrows}

\title{Polytopes with large transversal ratio}
\author{Michael Gene Dobbins\footnote{
		Department of Mathematics and Statistics, Binghamton University, NY, USA. \texttt{mdobbins@binghamton.edu}.}
	\ and   Seunghun Lee\footnote{
		Department of Mathematics, Keimyung University, Daegu, South Korea. Partially supported by the Institute for Basic Science (IBS-R029-C1). \texttt{seunghun.math@gmail.com}.} 
        }

\date{}

\usepackage[normalem]{ulem} 

\begin{document}
\maketitle
\begin{abstract}
The transversal ratio of a polytope $P$ is the minimum proportion of vertices of $P$ required to intersect each facet of $P$. The weak chromatic number of $P$ is the minimum number of colors required to color the vertices of $P$ so that no facet is monochromatic. We will construct an infinite family of $d$-polytopes for each $d\geq 5$ whose transversal ratio approaches 1 as the number of vertices grows. In particular, this implies that the weak chromatic number for $d$-polytopes is unbounded for each $d\geq 5$. 
The previous best known lower bounds on the supremum of the transversal ratio for $d$-polytopes for $d\geq 5$ were 2/5 for odd $d$ by Novik and Zheng, and 1/2 for even $d$ by Holmsen, Pach, and Tverberg. 
In the case of simplicial $(d-1)$-spheres, the best known lower bounds were 1/2 for $d=5$ and $6/11$ for $d=6$ by Novik and Zheng.
\end{abstract}

\section{Introduction}

\subsection{Statement of the main result}\label{subsec:statement}
For a hypergraph $H=(V, \mathcal{E})$, a \textit{transversal} is a subset of $V$ which intersects all the hyperedges in $\mathcal{E}$. The \textit{transversal number} of $H$, denoted by $\tau(H)$, is the minimum size of a transversal of $H$. The \textit{transversal ratio} of $H$ is the ratio $\rho(H)=\tau(H)/|V|$. 
Given a $d$-polytope $P$ on vertex set $V$, 
the hypergraph $H(P)$ associated to $P$ is 
\[H(P)=\{f\subseteq V:\textrm{$f$ is the vertex set of a facet of $P$}\}.\]

Our main result is the following. The proof is given in Section \ref{sec:proof}.

\begin{thm} \label{thm:main}
For $d\geq 5$ and $0 \leq r<1$, there is a simplicial $d$-polytope $P$ with $\rho(H(P)) \geq r$.
\end{thm}

This polytope is very large, with a number of vertices that grows roughly like the Ackermann function.
 
For a hypergraph $H=(V,\mathcal{E})$, a \textit{(weak) $k$-coloring} of $H$ is a coloring $c: V \to [k]$ so that every hyperedge in $\mathcal{E}$ attains at least two colors, where $[k]=\{1,2,\dots, k\}$. The \textit{(weak) chromatic number}, denoted by $\chi(H)$, is the minimum integer $k$ such that $H$ is $k$-colorable. Throughout the rest of the paper we omit the word ``weak'' when referring to coloring. 

Note that when $\chi(H)\leq k$, we have $\rho(H)\leq \frac{k-1}{k}$ by choosing $k-1$ color classes of smaller size. Thus Theorem \ref{thm:main} implies:

\begin{cor}
For $d\geq 5$ and a positive integer $N$, there is a simplicial $d$-polytope $P$ with $\chi(H(P)) \geq N$.
\end{cor}

\medskip 

\subsection{Ingredients} \label{subsec:ingredients}

The main ingredients of our construction are two-fold. One is the celebrated \textit{density Hales-Jewett theorem} by Furstenberg and Katznelson (see Theorem \ref{thm:density_Hales-Jewett}).  This is a fundamental result in Ramsey Theory which says that a large enough subset of a combinatorial cube must contain a combinatorial line \cite{density-Hales-Jewett_original} (see also \cite{density-Hales-Jewett_annals, density-Hales-Jewett_removal, density-Hales-Jewett_simple, Hales_Jewett_reasonable}). 
The other is the construction of a 5-polytope with a facet corresponding to each combinatorial line. 

The classical Hales-Jewett theorem \cite{hales-jewett} was used by 
Lee and Nevo in a similar way to show that the chromatic number of PL embeddable hypergraphs is unbounded \cite{lee_nevo2023colorings}. 
A limitation of their approach was that it involved extending a PL-embedding of a system of combinatorial lines to a PL sphere.  
This can be done using the extension theorem by Adiprasito and Pat\'akov\'a \cite{adiprasito-patakova2024higherdimensionalversionfarystheorem},\footnote{We also learned a proof of this from Lorenzo Venturello
and Geva Yashfe (personal communication).}
however, the extension requires additional vertices whose number cannot be controlled effectively. Consequently, the main result of \cite{lee_nevo2023colorings} has no immediate implication for the transversal ratios of polytopes, or even for that of simplicial spheres. Here we embed the combinatorial lines as facets of a 5-polytope without any additional vertices.

Our polytope construction involves embedding combinatorial lines on conics. 
Conics have also been used in other polytope constructions.  For example, Jürgen Richter-Gebert gave a new proof using conics that a 5-polytope constructed by Günter Ziegler has a non-prescribable 2-face \cite[6.5.(c)]{lectures_on_polytopes_book}, and obtained a \textit{non-Steinitz theorem}  in dimension 5 \cite[Example 3.4.3 and Chapter 15]{realization_space_polytope}.

\subsection{Background}\label{subsec:background}
Transversal number and chromatic number are fundamental concepts in combinatorics. Not only for abstract graphs and hypergraphs, those concepts have been also studied in connection to geometry; the examples include, but are not restricted to, $(p,q)$-theorems of combinatorial convexity (see \cite{discrete_geometry_matousek_book, combinatorial_convexity_barany, helly_type_survey_holmsen_wenger_handbook, helly_survey_amenta_de_Loera_soberon, helly_survey_barany_kalai} and references therein), stable sets and chromatic numbers of the graph of a flag sphere \cite{nevo_stable_flag, newman2023chromaticnumbersflag3spheres},  independence ratio of a graph with bounded genus \cite{genus_cutting}, coloring and choosability results for the face-hypergraph of a graph embedded in a surface \cite{coloring_surface_jctb, coloring_surface_answer, choosability_genus, Thomassen_list_coloring_2-sphere}, chromatic number of a triangulable $d$-dimensional manifold \cite{coloring_surface_lutz}, uniform hypergraphs embeddable in $\RR^d$ \cite{coloring_d-embeddable, lee_nevo2023colorings} and transversal ratio of stacked spheres \cite{cho2023transversal}. 

To the best of our knowledge, the earliest direct motivation to investigate transversal number of polytopes comes from the \textit{surrounding property} by Holmsen, Pach and Tverberg \cite{holmsen_surrounding}. Given a finite point set $P \subseteq \RR^d$ in general position with respect to the origin $O$, we say $P$ satisfies property $S(k)$, if any $k$-element subset of $P$ can be extended to a $(d+1)$-element subset of $P$ so that $O\in \conv(P)$. 
Applying the Gale transform to $P$, we obtain a dual point set $P^*$ in $\RR^{n-d-1}$.  There is an equivalent dual property $S^*(k)$ for $P^*$: among any $n-k$ points of $P^*$, there are some $n-d-1$ that form a facet of $\conv(P^*)$. When $P^*$ is in convex position, it means that the hypergraph $H(\conv(P^*))$ does not have a transversal of size $k$. By using the duality and considering transversal number of cyclic $d^*$-polytopes which is $\lceil(n-d^*)/2\rceil+1$  for even $d^*=n-d-1$, they obtained infinitely many point sets with $S(\lfloor d/2 \rfloor +1)$. 

This motivates studies of transversal ratio of polytopes, or simplicial spheres even more generally. The following question was originally asked by Andreas Holmsen in a personal communication.  Isabella Novik and Hailun Zheng \cite{novik_transversal_open_problem} also asked a similar, likely equivalent, question, see Subsection \ref{subsec:implication}. 
Let $H(\Delta)$ be the hypergraph consisting of the facets of a polytopal complex $\Delta$.
\begin{Q}\label{question_main}
For $d \geq 3$, let
\begin{align*}
    \rho_d^\mr{P}&:=\sup\{\rho(H(Q)):\textrm{$Q$ is a simplicial $d$-polytope}\}\textrm{, and}\\
    \rho_d^\mr{S}&:=\sup\{\rho(H(\Delta)):\textrm{$\Delta$ is a simplicial $(d-1)$-sphere}\}.
\end{align*}
What are the values of $\rho_d^\mr{P}$ and $\rho_d^\mr{S}$?
\end{Q}
In \cite{transversals_spheres_joseph_michael}, it was shown that $\rho_3^\mr{P}=\rho_3^\mr{S}=1/2$, and that $\rho_4^\mr{S}\geq 11/21$ where the constructions were obtained experimentally, and have larger transversal ratio than $1/2$, which is an upper bound for cyclic polytopes. Later, Novik and Zheng introduced novel constructions for Question \ref{question_main} which replaced the previous records for $d \geq 4$
\cite{novik_transversal_open_problem, novik2024transversalnumberssimplicialpolytopes}. They showed $\rho_4^\mr{S} \geq 4/7$, $\rho_5^\mr{S} \geq 1/2$, $\rho_6^\mr{S} \geq 6/11$, and $\rho_d^\mr{P} \geq 2/5$ for every odd $d\geq 5$.

\medskip 

\subsection{Implications and future directions} \label{subsec:implication}

\textbf{On Question \ref{question_main}.} Theorem \ref{thm:main} implies that $\rho_d^\mr{P}=\rho_d^\mr{S}=1$ for $d\geq 5$. This is rather surprising when we compare this with previous constructions for Question \ref{question_main}; many of them were obtained by utilizing the moment curve.
This choice looks natural; on one hand neighborly polytopes, including cyclic polytopes, have the largest number of facets so it is reasonable to expect to require more vertices to pierce all the facets. On the other hand, geometry on the moment curve can be described in a purely combinatorial way which is easy to handle.
When we use the moment curve however, as seen from the list of facets of a cyclic polytope, some parity issue on dimensions appears. Hence it was relatively easier to obtain constructions of large transversal ratio for even dimensional polytopes (or odd dimensional simplicial spheres) than for the others. Since the smallest dimension of our constructed polytopes is 5, this looks quite different from the previous ones in this aspect, and might suggest another way of constructing polytopes with large face numbers. 

An obvious open question remains for $d=4$ in Question \ref{question_main}. Our result and 3-dimensional spheres having unbounded chromatic number \cite{lee_nevo2023colorings} suggests the following conjecture.
\begin{con}
There are 4-polytopes requiring all but a vanishingly small portion of vertices for a transversal, 
i.e., 
$\rho_4^\mr{P}=\rho_4^\mr{S}=1$.
\end{con}

\begin{remark} \label{remark_eran}
    Our construction can be slightly modified to give the following conclusion: \textit{For $k$ and $d$ with $k\geq d$ and a given $d$-polytope $P$, let $\rho_k(P)$ be the proportion of vertices of $P$ needed to pierce all facets of $P$ with at least $k$ vertices each. For any $0\leq r<1$, there is a $d$-polytope $P$ with $\rho_k(P)\geq r$.} This motivates the following question.
\end{remark}
\begin{Q} \label{question_eran}
  Is there a $d$-polytope with no simplical facets and large transversal ratio?
\end{Q}

Note that the chromatic number of such a polytope must be unbounded, 
so the chromatic number of the graph of the polytope must also be unbounded.

\medskip 

\textbf{Chromatic number of geometrically embeddable hypergraphs and simplicial manifolds.} We say that a $k$-uniform hypergraph $H=(V,\mathcal{E})$ for $k\leq d+1$ is \textit{geometrically embeddable} in $\RR^d$ if there is an injection $\iota: V \to \RR^d$ such that for every hyperedge $h$ of $H$, $\conv(\iota(h))$ is a $(k-1)$-simplex, and for two distinct hyperedges $h_1$ and $h_2$, 
\[\conv(\iota(h_1))\cap \conv(\iota(h_1))   =\conv(\iota(h_1)\cap \iota(h_1)).\]
In \cite{coloring_d-embeddable}, the authors investigated the (weak) chromatic number of $k$-uniform hypergraphs geometrically embeddable in $\RR^d$ where $k$ and $d$ vary. This was an attempt for a generalization of the four color theorem in higher dimensions.

A particular question is whether the chromatic number is unbounded when the number of vertices grows to infinity, especially when $k=d+1$, see \cite[Section 4]{coloring_d-embeddable}. In \cite{lee_nevo2023colorings}, utilizing the hypergraphs by Ackerman, Keszegh and P\'alv\"olgyi \cite{ABAB_stabbed_pseudodisk} extended from earlier constructions in \cite{ABABA_construction_pach, abstract_polychromatic}, it was shown that when $k\leq d$, the chromatic number is unbounded. Also in the same paper, when $k=d+1$, a family of geometrically embeddable hypergraphs with chromatic number at least 3 were constructed for odd $d$. Here, a similar parity issue appears since the constructions were based on the moment curve.

By using Schlegel diagrams of our constructions, we obtain the following corollary.
\begin{cor}\label{cor:geometrically_embeddable}
For $d\geq 4$, there is an infinite family of $(d+1)$-uniform hypergraphs, which are geometrically embeddable in $\RR^d$, with unbounded (weak) chromatic number. 
\end{cor}
However our technique shows limitation for $\RR^3$. In fact, it does not look very clear if we can embed the combinatorial lines of a large combinatorial cube in $\RR^3$. Thus we propose the following question (see the Hales-Jewett hypergraph in Subsection \ref{subsec:density_HJ}).
\begin{Q} \label{question:HJ_embedding_R3}
Can we geometrically embed the Hales-Jewett hypergraph $\HJ(4,n)$ into $\RR^3$ for every positive integer $n$? 
\end{Q}

Our construction also gives a simpler construction of simplicial $d$-manifolds with unbounded chromatic numbers than the one given in \cite{lee_nevo2023colorings} as an extension of the previous study by Lutz and M{\o}ller \cite{coloring_surface_lutz}: one simply take the connected sum via facets of a given simplicial manifold and our construction.

\medskip

\textbf{Asymptotics of chromatic number and transversal ratio of simplicial spheres.} Since we can have arbitrarily large chromatic number and transversal ratio for $d$-polytopes for $d \geq 5$, it is reasonable to ask the following question:
\begin{Q} \label{question:growth}
For $d \geq 4$, let
\begin{align*}
    \rho_d^\mr{P}(n)&:= \max\{\rho(H(Q)):\textrm{$Q$ is a simplicial $d$-polytope on $n$ vertices}\}\textrm{, and}\\
    \rho_d^\mr{S}(n)&:= \max\{\rho(H(\Delta)):\textrm{$\Delta$ is a simplicial $(d-1)$-sphere on $n$ vertices}\}.
\end{align*} 
What is the asymptotic behavior of the functions 
$\rho_d^\mr{P}(n)$ and $\rho_d^\mr{S}(n)$? 
In particular, for $d \geq 5$, how quickly do $\rho_d^\mr{P}(n)$ and $\rho_d^\mr{S}(n)$ approach 1? 
\end{Q}

Novik and Zheng asked, what are the values of $\limsup_{n\to \infty}\rho_d^\mr{P}(n)$ and $\limsup_{n\to \infty}\rho_d^\mr{S}(n)$ \cite{novik_transversal_open_problem}.  
Their question is resolved for $d\geq 5$ by Theorem \ref{thm:main}, and remains open for $d=4$.
Question \ref{question:growth} is analogous to the question asked in \cite{coloring_d-embeddable} for geometrically embeddable hypergraphs; they not only asked for the asymptotic behavior, but asked also for the actual chromatic number as a function on $n$. Aside from the case of $d=4$, this looks like the natural next step to consider for higher dimensions.

In \cite{transversals_spheres_joseph_michael}, an upper bound $O(n^{(\lfloor d/2\rfloor-1)/(d-1)})$ for a chromatic number of a simplicial $(d-1)$-sphere on $n$ vertices was obtained. Later in \cite{novik2024transversalnumberssimplicialpolytopes}, Novik and Zheng obtained an upper bound of $n+1-\frac{1}{e}nm^{-1/d}$ on transversal number generally for pure $(d-1)$-dimensional complexes for $d \geq 2$ and $n$ sufficiently large 
where $m$ is the number of facets \cite[Theorem 3.1]{novik2024transversalnumberssimplicialpolytopes} as well as, a lower bound construction of pure $(d-1)$-dimensional complexes with $n$ vertices and $\Theta(m)$ facets whose transversal number is  $n-\Theta(n^{d/(d-1)}m^{-1/(d-1)})$ for $n$ and $m$ such that $1 \ll n \ll m \ll n^{(d+1)/2}$.  
In fact, the argument in  \cite[Lemma 6.3]{transversals_spheres_joseph_michael} used to obtain the upper bound on chromatic number can be slightly modified to obtain a tight upper bound which matches the lower bound by Novik and Zheng.
In particular for simplicial spheres, this gives an upper bound of $n-\Omega(n^{\lceil d/2\rceil/(d-1)})$ on the transveral number, which gives an upper bound of
$1-1/O(n^{(\lfloor d/2\rfloor-1)/(d-1)})$ 
on $\rho_d^\mr{P}(n)$ and $\rho_d^\mr{S}(n)$,  
by the upper bound theorems for polytopes and simplicial spheres by McMullen \cite{macmahon_counting_book} and Stanley \cite{upper_bound_stanley}, respectively.

This upper bound is very far from the lower bound that we expect from our construction. 
For $d=5$, 
we expect our construction to give a lower bound similar to 
$1-1/\Omega(\log_2^{***}(n))$\footnote{
$\log_b^*(n)$ is the number of times $\log_b$ must be applied to $n$ to obtain a value that is at most 1 
and $\log_b^{**}(n)$ is defined analogously for $\log_b^*$ et cetera. 
} 
compared to the upper bound of $1-1/O(n^{1/4})$.  
Specifically, we have $\rho_d^\mr{P}(n) \geq 1-\dhj_d^{-1}(\log_d(n))$ 
where $\dhj_d^{-1}(y) = \inf\{\varepsilon>0 : \dhj(d,\varepsilon) \leq y \}$ 
and $\dhj(d,\varepsilon)$ is the minimum quantity satisfying the conclusion of the density Hales-Jewett theorem (see Theorem \ref{thm:density_Hales-Jewett}). 
As both $d$ and $n$ grow, $\dhj_d^{-1}(n)$ is known to tend to 0 roughly at least as fast as the reciprocal of the inverse Ackermann function, 
but precise bounds for the density Hales-Jewett theorem have not been given \cite[Theorem 1.5]{density-Hales-Jewett_annals}.

\medskip

\textbf{On surrounding property.} 
Holmsen, Pach, and Tverberg 
asked for an arbitrarily large point set $P$ in $\RR^d$ in general position  with respect to the origin that satisfies the surrounding property $S(k)$ for large values of $k$ when the dimension $d$ is fixed.  Equivalently, they asked for a point set $P^*$ in $\RR^{d^*}=\RR^{n-d-1}$ in general position 
with convex hull having transversal number more than $k$ \cite[Problems 3 and 7]{holmsen_surrounding}. 
Theorem \ref{thm:main} gives a partial answer for this question.

\begin{cor}
    For any $0\leq r < 1$, there is a sufficiently large dimension $d$ and a point set $P$ in general position with respect to the origin in $\RR^d$ satisfying $S(rd)$.
\end{cor}

Even though Theorem \ref{thm:main} sheds some light on the problem on surrounding property, in fact our setting does not exactly fit into the problem since the primal dimension $d$ is fixed in \cite[Problems 3 and 7]{holmsen_surrounding} not the dual dimension $d^*$ for the space for transversal ratio. Rather, it is more about polytopes with ``small'' number of vertices in the following sense. The following problem seems to have a different flavor and has interesting features on its own right.
\begin{Q}
For a fixed positive integer $d$, what is the maximum transversal number of $H(P)$ for a simplicial $d^*$-polytope $P$ on $d^*+d+1$ vertices?
\end{Q}

\section{Proof of Theorem \ref{thm:main}} \label{sec:proof}

\subsection{Density Hales-Jewett theorem}
\label{subsec:density_HJ}
For positive integers $d$ and $n$, we define the hypergraph $\HJ(d,n)$ as follows. 
A \textit{combinatorial line} is a set of $n$ words in $[d]^n$ where the value at each position is either the same for each word or ranges over $[d]$ in a fixed order. 
That is, using an additional auxiliary element $*$, let $\tau=(\tau_1, \dots, \tau_n) \in ([d]\cup \{*\})^n$ be a sequence where $*$ appears at least once, or equivalently $\tau \in ([d]\cup \{*\})^n \setminus [d]^n$. 
For $k\in \RR$, let $\sigma(\tau,k)$ be the word obtained by substituting each instance of $*$ in $\tau$ with the value $k$, that is,
\begin{align*}
\sigma(\tau,k) &= (\sigma_1,\dots,\sigma_n) \quad \text{where}\quad  
\sigma_i = \begin{cases}
 k & \textrm{if $\tau_i=*$, and} \\
 \tau_i & \text{otherwise.}
\end{cases}
\end{align*}
The \textit{combinatorial line} $L_\tau$ 
is the set of all words obtained in this way for $k \in [d]$ from the sequence $\tau$, that is, 
\[L_\tau = \left\{\sigma(\tau,k) : k \in [d]\right\}.\]
The \textit{Hales-Jewett hypergraph}, $\HJ(d,n)$, is the hypergraph on ground set $[d]^n$ 
consisting of all combinatorial lines. 
That is,
\[\HJ(d,n)= \left([d]^n, \left\{L_\tau: \tau \in ([d]\cup \{*\})^n\setminus [d]^n \right\}\right).\]
We state the \textit{density Hales-Jewett theorem} using $\HJ(d,n)$ and  transversal ratio $\rho(\cdot)$. This celebrated theorem is a fundamental result of Ramsey Theory due to  Furstenberg and Katznelson \cite{density-Hales-Jewett_original}, and has several different proofs \cite{density-Hales-Jewett_annals, density-Hales-Jewett_removal, density-Hales-Jewett_simple}. 

\begin{thm}[Density Hales-Jewett Theorem \cite{density-Hales-Jewett_original}] \label{thm:density_Hales-Jewett}

The transversal ratio of $\HJ(d,n)$ approaches 1 as $n$ grows. 
That is, for every integer $d\geq 2$ and every $\varepsilon>0$, there is an integer $\dhj(d,\varepsilon)$ such that $\rho(\HJ(d,n))>1-\varepsilon$ for $n \geq \dhj(d,\varepsilon)$.

\end{thm}

By Theorem \ref{thm:density_Hales-Jewett}, it is sufficient to prove the following theorem in order to obtain Theorem \ref{thm:main}.

\begin{thm} \label{thm:realization_polytope}
For every positive integers $d \geq 5$ and $n$, there is a $d$-dimensional simplicial polytope $Q$ and a bijection $\varphi: [d]^n \to V(Q)$ 
such that for any combinatorial line $L$ in $\HJ(d,n)$, the image $\varphi(L)$ is the vertex set of some facet of $Q$. 
\end{thm}

\subsection{Proof of Theorem \ref{thm:realization_polytope}}
We prove Theorem \ref{thm:realization_polytope} in several steps.

\medskip 

\textbf{1. Drawing $\HJ(d,n)$ in the plane.} Choose vectors $v_1, \dots, v_n$ in $\RR^2$. We may assume that
\begin{enumerate}
      [label=\normalfont(\alph*)]
    \item  the $x$-component of $v_i$ is positive for every $i \in [n]$. \label{tag:x_positive} 
  \end{enumerate}
For every $\sigma=(\sigma_1, \dots, \sigma_n) \in \RR^n$, define a point
\[p_\sigma= \sum_{i=1}^n \sigma_i v_i.\]
By perturbation of the vectors $v_1, \dots, v_n$ if necessary, we may assume that $p_\sigma$ are all distinct for $\sigma \in [d]^n$. Let
\[P=\{p_\sigma: \sigma \in [d]^n\}.\]
By \ref{tag:x_positive}, the $x$-component of $p_\sigma$ is positive for every $\sigma \in [d]^n$. 

For a combinatorial line $L=L_\tau$, 
let 
\[
P_L = \{p_\sigma : \sigma\in L\},
\]
and let $S_L$ be the affine span of $P_L$. 
Equivalently, 
\[
 S_L = \left\{ p_{\sigma(\tau,t)} : t \in \RR  \right\}
\]
with $\sigma(\tau,t) \in \RR^n$ as in the definition of combinatorial line above. 
By \ref{tag:x_positive}, $S_L$ has a finite slope. Let 
\[y = a_Lx+b_L\] 
be the equation defining $S_L$. 
By perturbation of $v_1, \dots, v_n$, we may further assume that 
no other points of $P$ are collinear with the points of $P_L$; 
that is
\begin{enumerate}      [label=\normalfont(\alph*)]
    \setcounter{enumi}{1}
    \item $S_L \cap P=P_L$. \label{tag:l_int_V=L}
\end{enumerate}

\smallskip 

\textbf{2. Veronese embedding.} Define the affine version of the \textit{Veronese embedding} $\nu:\RR^2 \to \RR^5$ as
\[\nu(x,y)=(x^2,xy,y^2,x,y).\]

\begin{lemma}\label{lemmaConvex}
$\nu(P)$ is in convex position.
\end{lemma}

\begin{proof}
For $p \in P$, let $g_p$ denote the affine function from $\RR^5$ to $\RR$ with constant term $\|p\|^2$ and the same coefficients as the polynomial defined by the squared distance from $p$. 
That is, $g_p(x^2,xy,y^2,x,y) = \|(x,y)-p\|^2$.
Then, $g_p$ vanishes on $\nu(p)$ and is positive on the rest of $\nu(P)$. 
Hence, $p$ is a vertex of $\conv(\nu(P))$. 
\end{proof}

\begin{lemma}\label{lemma:line_support_bound}
For each combinatorial line $L$ of $\HJ(d,n)$, we can find an affine function $g_L$ from $\RR^5$ to $\RR$
that vanishes on $\nu(P_L)$ and is positive on $\nu(P\setminus P_L)$. 
Hence, $\conv(\nu(P_L))$ is a face of $\conv(\nu(P))$. 
\end{lemma}

\begin{proof}
Let $f_L(x,y)=(y-a_Lx-b_L)^2$, and 
let $g_L$ be the affine function with the same coefficients and constant as $f_L$.
That is, $g_L(x^2,xy,y^2,x,y) = f_L(x,y)$. 
Then, $g_L(\nu(p))=f_L(p)=0$ for every point $p\in P_L$,  
and $g_L(\nu(p))=f_L(p)>0$ for every $p \in P \setminus P_L$
by our assumption \ref{tag:l_int_V=L}. 
\end{proof}

\smallskip

\textbf{3. Making combinatorial lines into facets in $\RR^5$.} 
We make a suitable perturbation $P^\epsilon$ of $P$ so that $\conv(\nu(P_L^\epsilon))$ is a facet. 
For $\epsilon>0$, let $P^\epsilon$ be the result of translating each point of $P$ upward by a distance proporational to the square root of the $x$-coordinate, 
\[P^\epsilon = \left\{  \phi_\epsilon(p) : p \in P  \right\}
\quad \text{with} \quad 
\phi_\epsilon(x,y) = (x, y+\sqrt{\epsilon x}).
\]

\definecolor{ccqqqq}{rgb}{1,0,0}
\definecolor{rvwvcq}{rgb}{0.08235294117647059,0.396078431372549,0.7529411764705882}
\definecolor{qqwxvy}{rgb}{0,0.403921568627451,0.34509803921568627}
\definecolor{wewdxt}{rgb}{0.43137254901960786,0.42745098039215684,0.45098039215686275}
\definecolor{wvvxds}{rgb}{0.396078431372549,0.3411764705882353,0.8235294117647058}
\definecolor{ffqqtt}{rgb}{1,0,0.2}
\definecolor{cxvqqq}{rgb}{0.7803921568627451,0.3137254901960784,0}

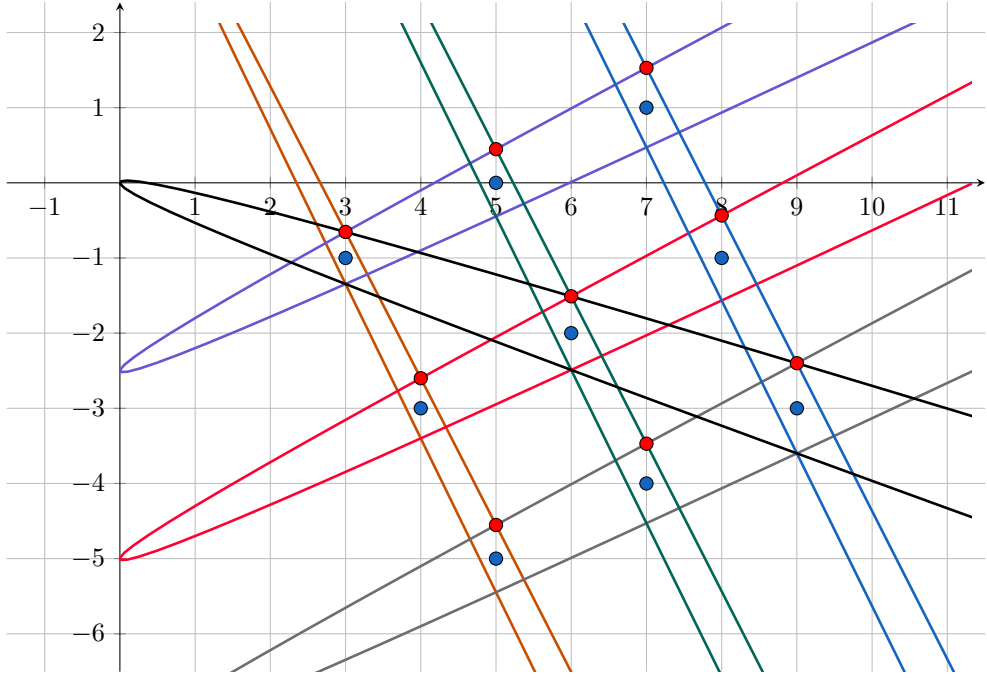
\begin{figure}[ht]
\centering
\begin{tikzpicture}[line cap=round,line join=round,>=triangle 45,x=1cm,y=1cm]
\begin{axis}[
x=1cm,y=1cm,
axis lines=middle,
ymajorgrids=true,
xmajorgrids=true,
xmin=-1.5,
xmax=11.5,
ymin=-6.5,
ymax=2.4,
xtick={-1,0,...,11},
ytick={-6,-5,...,2},]
\clip(-3.0264529691201787,-6.583060220621225) rectangle (11.326946962807709,2.114284182575456);
\draw [samples=50,rotate around={-153.43494882292202:(0.0016,5.0048)},xshift=0.0016cm,yshift=5.0048cm,line width=1pt,color=cxvqqq,domain=-0.2897944098839684:0.2897944098839684)] plot (\x,{(\x)^2/2/0.001788854381999805});
\draw [samples=50,rotate around={-63.43494882292201:(0.0016,-5.0072)},xshift=0.0016cm,yshift=-5.0072cm,line width=1pt,color=ffqqtt,domain=-0.8300284332479095:0.8300284332479095)] plot (\x,{(\x)^2/2/0.01431083505599844});
\draw [samples=50,rotate around={-63.43494882292201:(0.0016,-2.5072)},xshift=0.0016cm,yshift=-2.5072cm,line width=1pt,color=wvvxds,domain=-0.8300284332479198:0.8300284332479198)] plot (\x,{(\x)^2/2/0.014310835055998616});
\draw [samples=50,rotate around={-63.43494882292201:(0.0016000000005798043,-7.50719999999971)},xshift=0.0016000000005798043cm,yshift=-7.50719999999971cm,line width=1pt,color=wewdxt,domain=-0.83002843324793:0.83002843324793)] plot (\x,{(\x)^2/2/0.014310835055998793});
\draw [samples=50,rotate around={-153.43494882292202:(0.0015999999997591274,10.00480000000048)},xshift=0.0015999999997591274cm,yshift=10.00480000000048cm,line width=1pt,color=qqwxvy,domain=-0.34703775010792764:0.34703775010792764)] plot (\x,{(\x)^2/2/0.0017888543819996271});
\draw [samples=50,rotate around={-153.43494882292202:(0.0016000000007654336,15.004799999998468)},xshift=0.0016000000007654336cm,yshift=15.004799999998468cm,line width=1pt,color=rvwvcq,domain=-0.39354796403999615:0.39354796403999615)] plot (\x,{(\x)^2/2/0.0017888543819999825});
\draw [samples=50,rotate around={-108.43494882292201:(0.0009,0.0057)},xshift=0.0009cm,yshift=0.0057cm,line width=1pt,domain=-0.8879675669752809:0.8879675669752809)] plot (\x,{(\x)^2/2/0.01707629936490925});
\begin{scriptsize}
\draw[color=cxvqqq] (0.23736923762311357,5.188025834664687);
\draw[color=ffqqtt] (0.16319146019712966,-5.020690783586427);
\draw[color=wvvxds] (0.16319146019712966,-2.5264630176376985);
\draw[color=wewdxt] (0.16319146019712966,-7.524190771713403);
\draw[color=qqwxvy] (0.20028034891012161,10.185753588740392);
\draw[color=rvwvcq] (0.23736923762311357,15.183481342816096);
\draw [fill=rvwvcq] (3,-1) circle (2.5pt);
\draw[color=rvwvcq] (3.0746692241669984,-0.8018296924835591);
\draw [fill=rvwvcq] (5,0) circle (2.5pt);
\draw[color=rvwvcq] (5.077469214668564,0.19957030276723156);
\draw [fill=rvwvcq] (7,1) circle (2.5pt);
\draw[color=rvwvcq] (7.070996982991882,1.2009702980180221);
\draw [fill=rvwvcq] (6,-2) circle (2.5pt);
\draw[color=rvwvcq] (6.0695969877410985,-1.8032296877343494);
\draw [fill=rvwvcq] (8,-1) circle (2.5pt);
\draw[color=rvwvcq] (8.072396978242665,-0.8018296924835591);
\draw [fill=rvwvcq] (5,-5) circle (2.5pt);
\draw[color=rvwvcq] (5.077469214668564,-4.798157451308474);
\draw [fill=rvwvcq] (7,-4) circle (2.5pt);
\draw[color=rvwvcq] (7.070996982991882,-3.7967574560576827);
\draw [fill=rvwvcq] (9,-3) circle (2.5pt);
\draw[color=rvwvcq] (9.073796973493447,-2.8046296829851403);
\draw [fill=rvwvcq] (4,-3) circle (2.5pt);
\draw[color=rvwvcq] (4.076069219417781,-2.8046296829851403);
\draw[color=black] (0.19100812673187362,0.24593141365847185);
\draw [fill=ccqqqq] (3,-0.6535898384862175) circle (2.5pt);
\draw[color=ccqqqq] (3.0746692241669984,-0.45875747188838073);
\draw [fill=ccqqqq] (5.000000000006366,0.4472135955034251) circle (2.5pt);
\draw[color=ccqqqq] (5.077469214668564,0.6446369673231385);
\draw [fill=ccqqqq] (7.000000000001204,1.529150262213568) circle (2.5pt);
\draw[color=ccqqqq] (7.070996982991882,1.7294869621781617);
\draw [fill=ccqqqq] (4,-2.6) circle (2.5pt);
\draw[color=ccqqqq] (4.076069219417781,-2.3966519071422256);
\draw [fill=ccqqqq] (6,-1.5101020514433705) circle (2.5pt);
\draw[color=ccqqqq] (6.0695969877410985,-1.3118019122872022);
\draw [fill=ccqqqq] (8.000000000003226,-0.4343145750490342) circle (2.5pt);
\draw[color=ccqqqq] (8.072396978242665,-0.23622413961042726);
\draw [fill=ccqqqq] (9,-2.4) circle (2.5pt);
\draw[color=ccqqqq] (9.073796973493447,-2.201935241399016);
\draw [fill=ccqqqq] (7.0000000000072236,-3.4708497377832024) circle (2.5pt);
\draw[color=ccqqqq] (7.070996982991882,-3.268240791897543);
\draw [fill=ccqqqq] (4.999999999999837,-4.552786404499722) circle (2.5pt);
\draw[color=ccqqqq] (5.077469214668564,-4.353090786752567);
\end{scriptsize}
\end{axis}
\end{tikzpicture}
\caption{A drawing of the vertices of $\HJ(3,2)$ 
and the curve $Z_L^\epsilon$ for each combinatorial line $L$ of $\HJ(3,2)$. The blue dots are points of $P$ and the red dots are points of $P^\epsilon$.}
   \label{fig:perturbation1}
\end{figure}

An important consequence of this perturbation, is that points 
$P^\epsilon_L = \{\phi_\epsilon(p) : p \in P_L\}$ now lie on a unique conic for each combinatorial line $L$; see Figure \ref{fig:perturbation1}. 
Specifically, 
let \[f_L^\epsilon(x,y):=f_L(x,y)-\epsilon x=(y-a_Lx-b_L)^2-\epsilon x,\]
and let $Z_L^\epsilon = [f_L^\epsilon]^{-1}(0)$ be the zero set of $f_L^\epsilon$.  
Observe that $Z_L^\epsilon$ is a parabola, and appears only in the region $x\geq 0$. 

\begin{lemma}\label{lemmaPZ}
$P^\epsilon_L \subset Z^\epsilon_L$.
\end{lemma}

\begin{proof}
If $p = (x,y) \in P^\epsilon_L$, then $y = a_Lx+b_L +\sqrt{\epsilon x}$ and $x > 0$ by \ref{tag:x_positive}, 
so $f_L^\epsilon(p) = \sqrt{\epsilon x}^2 -\epsilon x = 0$. 
\end{proof}

Let $g^\epsilon_L$ be the affine function with the same coefficients and constant as $f^\epsilon_L$
so that $g^\epsilon_L(x^2,xy,y^2,x,y) = f^\epsilon_L(x,y)$, 
and let us choose $\epsilon$ according to the following lemma. 

\begin{lemma} \label{lemma:parabola_support_bound}
   We can find a sufficiently small $\epsilon>0$ so that the affine function
   $g^\epsilon_L$ vanishes on $\nu(P_L^\epsilon)$ and is positive on $\nu(P^\epsilon\setminus P_L^\epsilon)$.
\end{lemma}
\begin{proof}
Observe that $g^\epsilon_L$ vanishes on $\nu(Z^\epsilon_L)$ by definition, 
so in particular $g^\epsilon_L$ vanishes on $\nu(P^\epsilon_L)$ by Lemma \ref{lemmaPZ}.
Since $P$ is finite, the largest $x$-coordinate of $P$ is bounded, so we can make $\nu(P^\epsilon)$ arbitrarily close to $\nu(P)$ by choosing $\epsilon$ small enough, and we can make $f_L^\epsilon$ arbitrarily close to $f_L = f^0_L$. 
Hence, $g^\epsilon_L$ is positive on the rest of $P^\epsilon$ for $\epsilon$ sufficiently small by Lemma \ref{lemma:line_support_bound}. 
\end{proof}

\begin{lemma}\label{lemma_general_position_R5}
For each combinatorial line $L$, any 5 points of $\nu(P_L^\epsilon)$ are affinely independent.    
Hence, $\conv(\nu(P_L^\epsilon))$ is a facet of $\conv(\nu(P^\epsilon))$. 
\end{lemma}

\begin{proof}
Since $P_L^\epsilon$ lies on the parabola $Z^\epsilon_L$, no 3 points of $P_L^\epsilon$ are collinear. Hence, any 5 points of $P_L^\epsilon$ determine a unique conic, which is $Z^\epsilon_L$. 
This implies that the image of these 5 points by $\nu$ lie on a unique hyperplane $h^\epsilon_L=[g^\epsilon_L]^{-1}(0)$, and so the affine span of the 5 points must be the hyperplane $h^\epsilon_L$. 
Thus, the 5 points are affinely independent. 
The second part of the lemma now follows by Lemma \ref{lemma:parabola_support_bound}.
\end{proof}

\textbf{4. Making combinatorial lines into facets in $\RR^d$.} 
Suppose $d>5$.
We will perturb points orthogonally to $\RR^5$ as a subspace of $\RR^d$. 
For each $p_{\sigma}^\epsilon = \phi_\epsilon(p_\sigma) \in P^\epsilon$, 
let 
\[
p_\sigma^\mr{z} = (\nu(p_{\sigma}^\epsilon),z_{\sigma,1},\dots,z_{\sigma,d-5}) \in \RR^d, \quad 
P^\mr{z} = \{ p_\sigma^\mr{z} : \sigma \in [d]^n\}, \quad 
P^\mr{z}_L = \{ p_\sigma^\mr{z} : \sigma \in L\} 
\]
where the $z_{\sigma,j}$ are chosen according to the following lemma.

\begin{lemma}
We can choose $z_{\sigma,j}$ such that the points of $P_L^\mr{z}$ are affinely independent for each combinatorial line $L$.
Moreover, $\conv(P_L^\mr{z})$ is a facet of $\conv(P^\mr{z})$ with this choice.
\end{lemma}

\begin{proof}
For each combinatorial line $L = \{\sigma_1,\dots,\sigma_d\}$, let 
\[A_L=\begin{pmatrix}
1 & 1 & \cdots & 1\\
p_{\sigma_1}^\mr{z} & p_{\sigma_2}^\mr{z} & \cdots & p_{\sigma_d}^\mr{z}\\
\end{pmatrix} \in \RR^{(d+1)\times d}.\]
Since the points $\nu(p_{\sigma_1}^\epsilon),\dots,\nu(p_{\sigma_d}^\epsilon)$ are contained in the unique hyperplane $h_L^\epsilon$,  the first 6 rows of $A_L$ form a submatrix of rank 5, 
so the remaining entries of $A_L$ can be chosen so that $A_L$ has full rank, 
which gives $\det(A_L^\mr{T}A_L) \neq 0$. 
Hence, $\det(A_L^\mr{T}A_L)$ is a non-zero polynomial in variables $z_{\sigma,j}$, 
so the zero set of $\det(A_L^\mr{T}A_L)$ has measure zero in $\RR^{d^n(d-5)}$.
Therefore, we may choose values for the $z_{\sigma,j}$ where these polynomials are non-zero for all combinatorial lines. 
Thus, $P_L^\mr{z}$ is affinely independent for each combinatorial line $L$.

For the second part of the lemma, 
let us extend $g_L^\epsilon$ to an affine function $g_L^\mr{z}: \RR^d \to \RR$ by coordinate projection.
That is, $g_L^\mr{z} = g_L^\epsilon\circ\pi$ where $\pi:\RR^d \to \RR^5$ is the projection to the first 5 coordinates.
Since the first 5 coordinates of $p_\sigma^\mr{z}$ are $\nu(p_\sigma^\epsilon)$, 
we have $g_L^\mr{z}(p_\sigma^\mr{z}) = g_L^\epsilon(\nu(p_\sigma^\epsilon))$, 
so $g_L^\mr{z}$ vanishes on $P_L^\mr{z}$ and is positive on the rest of $P^\mr{z}$. 
\end{proof}

\textbf{5. A generic perturbation.} 
Let $P^\mr{s}$ be a small enough generic perturbation of the point set $P^\mr{z}$ (with $P^\mr{z}=\nu(P^\epsilon)$ in the case of $d=5$) for the following three conditions. 
Observe that the points of $\nu(P^\epsilon) = \pi(P^\mr{z})$ are in convex position 
where $\pi:\RR^d \to \RR^5$ is the projection to the first 5 coordinates  
by the same argument as Lemma \ref{lemmaConvex}, 
so $P^\mr{z}$ is in convex position. 
First, choose a small enough perturbation so that $P^\mr{s}$ is in convex position.
Second, the set of points $P^\mr{s}_L$ corresponding to a combinatorial line $L\subset [d]^n$ remain affinely independent 
with the correspondence inherited from $P^\mr{z}$.
Third, an affine function $g_L$ that is slightly perturbed from $g_L^\mr{z}$ vanishes on $P^\mr{s}_L$ and is positive on the rest of $P^\mr{s}$. 
Hence, $Q_L = \conv(P^\mr{s}_L)$ is a simplicial facet of $Q = \conv(P^\mr{s})$. 
Furthermore, since the perturbation is chosen generically, $Q$ is simplicial.

This completes the proof of Theorem \ref{thm:realization_polytope}.  \qed
\section*{Acknowledgement}
We thank Andreas Holmsen for introducing Question \ref{question_main} and the surrounding property, Eran Nevo for his suggestion to consider embedding linear hypergraphs into $\RR^d$ and noticing Remark \ref{remark_eran} and Question \ref{question_eran}, 
Isabella Novik and Varun Shah for discussion on Question \ref{question_eran}, and Seonghyuk Im and Hong Liu for providing references regarding bounds on the density Hales-Jewett function.

\bibliographystyle{hplain}
\bibliography{bibliography}

\end{document}